\documentclass[12pt,a4paper]{article}
\usepackage{indentfirst}
\usepackage{amsmath,amssymb,amsfonts,amsthm,graphics}
\usepackage{makeidx}
\usepackage{ifpdf}
\newtheorem{thm}{Theorem}

\newtheorem{lem}{Lemma}

\theoremstyle{definition}

\def\-{\mbox{--}}

\newtheorem{pro}{Proposition}

\newtheorem{obs}{Observation}

\begin{document}
\title{\bf The 3-rainbow index of a graph\Large\bf \footnote{Supported by NSFC No.11071130.}}
\author{\small Lily Chen, Xueliang~Li, Kang Yang, Yan~Zhao\\
\small Center for Combinatorics and LPMC-TJKLC\\
\small Nankai University, Tianjin 300071, China\\
\small lily60612@126.com; lxl@nankai.edu.cn; \\
\small yangkang@mail.nankai.edu.cn; zhaoyan2010@mail.nankai.edu.cn}
\date{}
\maketitle
\begin{abstract}

Let $G$ be a nontrivial connected graph with an edge-coloring $c:
E(G)\rightarrow \{1,2,\ldots, q\},$ $q \in \mathbb{N}$, where
adjacent edges may be colored the same. A tree $T$ in $G$ is a
$rainbow~tree$ if no two edges of $T$ receive the same color. For a
vertex subset $S\subseteq V(G)$, a tree that connects $S$ in $G$ is
called an $S$-tree. The minimum number of colors that are needed in
an edge-coloring of $G$ such that there is a rainbow $S$-tree for
each $k$-subset $S$ of $V(G)$ is called $k$-rainbow index, denoted
by $rx_k(G)$. In this paper, we first determine the graphs whose
$3$-rainbow index equals $2$, $m,$ $m-1$, $m-2$, respectively. We
also obtain the exact values of $rx_3(G)$ for regular complete
bipartite and multipartite graphs and wheel graphs. Finally, we give
a sharp upper bound for $rx_3(G)$ of $2$-connected graphs and 2-edge
connected graphs, and graphs whose $rx_3(G)$ attains the upper bound
are characterized.



\end{abstract}

\section{Introduction}
All graphs considered in this paper are simple, finite and
undirected. We follow the terminology and notation of Bondy and
Murty \cite{Bondy}. Let $G$ be a nontrivial connected graph with an
edge-coloring $c: E(G)\rightarrow \{1,2,\ldots,q\},$ $q \in
\mathbb{N}$, where adjacent edges may be colored the same. A path of
$G$ is a \emph{rainbow path} if no two edges of the path are colored
the same. The graph $G$ is \emph{rainbow connected} if for every two
vertices $u$ and $v$ of $G$, there is a rainbow path connecting $u$
and $v$. The minimum number of colors for which there is an
edge-coloring of $G$ such that $G$ is rainbow connected is called
the \emph{rainbow connection number} of $G$, denoted by $rc(G)$.
Results on the rainbow connections can be found in
\cite{Caro,Chartrand1,Chartrand,LS, Sun}.

These concepts were introduced by Chartrand et al. in
\cite{Chartrand1}. In \cite{Zhang}, they generalized the concept of
rainbow path to rainbow tree. A tree $T$ in $G$ is a $rainbow~tree$
if no two edges of $T$ receive the same color. For $S\subseteq V
(G)$, a $rainbow\ $S-$tree$ is a rainbow tree that connects the
vertices of $S$. Given a fixed integer $k$  with $2\leq k \leq n$,
the edge-coloring $c$ of $G$ is called a $k$-$rainbow~coloring$ if
for every $k$-subset $S$ of  $V(G)$, there exists a rainbow
$S$-tree. In this case, $G$ is called $k$-$rainbow~connected$. The
minimum number of colors that are needed in a $k$-$rainbow~coloring$
of $G$ is called the $k$-$rainbow~index$ of $G$, denoted by
$rx_k(G)$. Clearly, when $k=2$, $rx_2(G)$ is nothing new but the
rainbow connection number $rc(G)$ of $G$. For every connected graph
$G$ of order $n$, it is easy to see that $rx_2(G)\leq rx_3(G)\leq
\cdots \leq rx_n(G)$.

The $Steiner~distance$ $d(S)$ of a subset $S$ of vertices in $G$ is
the minimum size of a tree in $G$ that connects $S$. Such a tree is
called a $Steiner$ $S$-$tree$ or simply a $Steiner~tree$. The
$k$-$Steiner~diameter$ $sdiam_k(G)$ of $G$ is the maximum Steiner
distance of $S$ among all $k$-subsets $S$ of $G$. Then there is a
simple upper bound and a lower bound for $rx_k(G)$.

\begin{obs}[\cite{Zhang}]\label{obs1}
For every connected graph $G$ of order $n\geq 3$ and each integer $k$ with $3\leq k\leq n$,
$k-1\leq sdiam_k(G)\leq rx_k(G)\leq n-1$.
\end{obs}

They showed that trees are composed of a class of graphs whose
$k$-rainbow index attains the upper bound.
\begin{pro}[\cite{Zhang}]\label{pro1}
Let $T$ be a tree of order $n\geq 3$. For each integer $k$ with $3\leq k\leq n$, $rx_k(T)=n-1$.
\end{pro}

Before showing Proposition \ref{pro1}, they gave the following observation.
\begin{obs}[\cite{Zhang}]\label{obs2}
Let $G$ be a connected graph of order $n$ containing two bridges $e$
and $f$. For each integer $k$ with $2\leq k\leq n$, every
$k$-rainbow coloring of $G$ must assign distinct colors to $e$ and
$f$.
\end{obs}

For $k=2$, $rx_2(G)=rc(G)$, which has been studied extensively, see
\cite{LS, Sun}. But for $k\geq 3$, very few results has been
obtained. In this paper, we focus on $k=3$. By Observation
\ref{obs1}, we have $rx_3(G)\geq 2$. On the other hand, if $G$ is a
nontrivial connected graph of size $m$, then the coloring that
assigns distinct colors to the edges of $G$ is a $3$-rainbow
coloring, hence $rx_3(G)\leq m$. So we want to determine the graphs
whose 3-rainbow index equals the values $2,$ $m$, $m-1$ and $m-2$,
respectively. The following results are needed.

\begin{lem}[\cite{Zhang}]\label{lem1}
For $3\leq n\leq 5$, $rx_3(K_n)=2$.
\end{lem}

\begin{lem}[\cite{Zhang}]\label{lem2}
Let $G$ be a connected graph of order $n\geq 6$. For each integer
$k$ with $3\leq k\leq n$, $rx_k(G)\geq 3$.
\end{lem}

\begin{thm}[\cite{Zhang}]\label{thm1}
For each integer $k$ and $n$ with $3\leq k\leq n$, \begin{equation}
 rx_k(C_n)=
   \begin{cases}
      n-2, & \text{if $k=3$ and $n\geq4$}; \\
     n-1, & \text{if $k=n=3$ or $4\leq k\leq n$}.
    \end{cases}
\end{equation}
\end{thm}

\begin{thm}[\cite{Zhang}]\label{thm2}
If $G$ is a unicyclic graph of order $n\geq 3$ and girth $g\geq 3$, then
\begin{equation}
 rx_k(G)=
   \begin{cases}
      n-2, & \text{$k=3$ and $g\geq4$}; \\
     n-1, & \text{$g=3$ or $4\leq k\leq n$}.
    \end{cases}
\end{equation}
\end{thm}

The following observation is easy to verify.
\begin{obs}\label{obs3}
Let $G$ be a connected graph and $H$ be a connected spanning subgraph of $G$. Then $rx_3(G)\leq rx_3(H)$.
\end{obs}

In Section 2, we are going to determine the graphs whose $3$-rainbow
index equals the values $2$, $m$, $m-1$ or $m-2$, respectively. In
Section 3, we will determine the $3$-rainbow~index for the complete
bipartite graphs $K_{r,r}$ and complete $t$-partite graphs $K_r^t$
as well as the wheel $W_n$. Finally, we will give a sharp upper
bound of $rx_3(G)$ for $2$-connected graphs and 2-edge connected
graphs, and graphs whose $3$-rainbow~index attains the upper bound
are characterized.

\section{Graphs with $rx_3(G)\bf{=2, m, m-1, m-2}$}

From Lemma \ref{lem2}, if $rx_3(G)=2$, then the order $n$ of $G$
satisfies $3\leq n\leq 5$.

\begin{thm}\label{thm3}
Let $G$ be a connected graph of order $n$. Then $rx_3(G)=2$ if and
only if $G=K_5$ or $G$ is a 2-connected graph of order 4 or $G$ is
of order 3.
\end{thm}

\begin{proof}
If $n=3$, it is easy to see that $rx_3(G)=2$.

If $n=4$, assume that $G$ is not 2-connected, then there is a cut
vertex $v$. It is easy to see that a tree connecting $V(G)\setminus
v$ has size 3, thus $rx_3(G)\geq 3$, a contradiction.

If $n=5$, let $V(G)=\{v_1,v_2,v_3,v_4,v_5\}$. Assume that
$rx_3(G)=2$ but $G$ is not $K_5$. Let $c: E(G)\rightarrow \{1,2\}$
be a rainbow coloring of $G$. Since every three vertices belong to a
rainbow path of length 2, there is no monochromatic triangle. Now we
show that the maximum degree $\Delta(G)$ is 4. If $\Delta(G)$ is 2,
then $G$ is a cycle or a path, and it is easy to check that
$rx_3(G)$ is 3 or 4, a contradiction. Assume that $\Delta(G)$ is 3.
Let $d(v_1)=3$ and $N(v_1)=\{v_2,v_3,v_4\}$. Then at least two edges
incident to $v_1$ have the same color, say $c(v_1v_2)=c(v_1v_3)=1$.
Consider $\{v_1,v_2,v_5\}$, $\{v_1,v_3,v_5\}$, this forces
$c(v_2v_5)=c(v_3v_5)=2$. Consider $\{v_1,v_2,v_3\}$, it implies that
$c(v_2v_3)=2$, but now $\{v_2,v_3,v_5\}$ forms a monochromatic
triangle, a contradiction. Thus $\Delta(G)=4$. Suppose $d(v_1)=4$.
If there are three edges incident to $v_1$ colored the same, say
$c(v_1v_2)=c(v_1v_3)=c(v_1v_4)=1$, then consider the three vertices
$v_2$, $v_3$ and $v_4$. Since these three vertices must belong to a
rainbow path of length $2$, without loss of generality, assume that
$c(v_2v_3)=1$ and $c(v_3v_4)=2$. However then $\{v_1,v_2,v_3\}$ is a
monochromatic triangle, which is impossible.  Therefore only two
edges incident to $v_1$ are assigned the same color. Since $G$ is
not $K_5$, $G$ is a spanning subgraph of $K_5-e$. Since $d(v_1)=4$,
we may assume that $G$ is a spanning subgraph of $K_5-v_3v_4$. Let
$G'=K_5-v_3v_4$. Consider $\{v_1,v_3,v_4\}$, $v_1v_3$ and $v_1v_4$
must have different colors, without loss of generality, assume that
$c(v_1v_3)=1$ and $c(v_1v_4)=2$. By symmetry, suppose $c(v_1v_2)=1$
and $c(v_1v_5)=2$. Then $c(v_2v_3)=2$, $c(v_4v_5)=1$. Consider
$\{v_2,v_3,v_4\}$, $\{v_3,v_4,v_5\}$, $\{v_2,v_3,v_5\}$, then
$c(v_2v_4)=1$, $c(v_3v_5)=2$, $c(v_2v_5)=1$, but now
$\{v_2,v_4,v_5\}$ forms a monochromatic triangle, which is
impossible. Hence, $rx_3(G)\geq rx_3(G')\geq 3$, contradicting to
the assumption.
\end{proof}

\begin{thm}\label{thm4}
(1) $rx_3(G)=m$ if and only if $G$ is a tree.

(2) $rx_3(G)=m-1$ if and only if $G$ is a unicyclic graph with girth 3.

(3) $rx_3(G)=m-2$ if and only if $G$ is a unicyclic graph with girth at least 4.
\end{thm}

\begin{proof}

(1). By Proposition \ref{pro1}, if $G$ is a tree, then
$rx_3(G)=n-1=m$. Conversely, if $rx_3(G)=m$ but $G$ is not a tree,
then $m\geq n$. By Observation \ref{obs1}, $rx_3(G)\leq n-1\leq
m-1$, a contradiction.

(2). If $G$ is a unicyclic graph with girth $3$, by Theorem
\ref{thm2}, $rx_3(G)=n-1=m-1$. Conversely, if $rx_3(G)=m-1$, then by
(1), $G$ must contain cycles. If $G$ contains at least two cycles,
then $m\geq n+1$. By Observation \ref{obs1}, $rx_3(G)\leq n-1\leq
m-2$, a contradiction. Thus, $G$ contains exactly one cycle. If the
cycle of $G$ is of length at least 4, then by Theorem \ref{thm2},
$rx_3(G)=n-2=m-2$, a contradiction. Thus, the cycle of $G$ is of
length 3, the result holds.

(3). If $G$ is a unicyclic graph with girth at least $4$, by Theorem
\ref{thm2}, $rx_3(G)=n-2=m-2$. Conversely, if $rx_3(G)=m-2$ and
$m\geq n+2$, then by Observation \ref{obs1}, $rx_3(G)\leq n-1\leq
m-3$, a contradiction. Thus, $m\leq n+1$. If $m=n$, then $G$ is a
unicyclic graph. By Theorem \ref{thm2}, the girth of $G$ is at least
4. If $m=n+1$, and there are two edge-disjoint cycles $C_1$ and
$C_2$ of length $g_1$ and $g_2$ such that $g_1\geq g_2$, then if
$g_1\geq 4$, we assign $g_1-2$ colors to $C_1$, $g_2-1$ new colors
to $C_2$ and assign new distinct colors to all the remaining edges,
which make $G$ 3-rainbow connected, hence $rx_3(G)\leq m-3$, a
contradiction. Therefore $g_1=g_2=3$. In this case, we assign each
cycle with three colors 1, 2, 3, and assign new colors to all the
remaining edges, then $G$ is 3-rainbow connected, thus $rx_3(G)\leq
m-3$. If these two cycles are not edge-disjoint, we can also use
$m-3$ colors to make $G$ 3-rainbow connected, a contradiction.
\end{proof}

\section{The 3-rainbow index of some special graphs}

In this section, we determine the $3$-rainbow index of some special
graphs. First, we consider the regular complete bipartite graphs
$K_{r,r}$. It is easy to see that when $r=2$, $rx_3(K_{2,2})=2$ and,
logically, we can define $rx_3(K_{1,1})=0$.

\begin{thm}\label{thm5}
For integer $r$ with $r\geq 3$, $rx_3(K_{r,r})=3$.
\end{thm}

\begin{proof}
Let $U$ and $W$ be the partite sets of $K_{r,r}$, where $|U|=|W|=r$.
Suppose that $U=\{u_1,\cdots,u_r\}$, $W=\{w_1,\cdots,w_r\}$. If
$S\subseteq U$ and $|S|=3$, then every $S$-tree has size at least 3,
hence $rx_3(K_{r,r})\geq 3$.

Next we show that $rx_3(K_{r,r})\leq 3$.  We define a coloring $c$:
$E(K_{r,r})\rightarrow \{1,2,3\}$ as follows.
 \begin{equation}
c(u_iw_j)=
   \begin{cases}
      1, & 1\leq i=j\leq r; \\
     2, & 1\leq i< j\leq r; \\
    3, & 1\leq j< i\leq r.
    \end{cases}
\end{equation}

Now we show that $c$ is a $3$-rainbow coloring of $K_{r,r}$. Let $S$
be a set of three vertices of $K_{r,r}$. We consider two cases.

{\bf Case 1.}~~The vertices of $S$ belong to the same partition of
$K_{r,r}$. Without loss of generality, let $S=\{u_i,u_j,u_k\}$,
where $i< j< k$. Then $T=\{u_iw_j,u_jw_j,u_kw_j\}$ is a rainbow
$S$-tree.

{\bf Case 2.}~~The vertices of $S$ belong to different partitions of
$K_{r,r}$. Without loss of generality, let $S=\{u_i,u_j,w_k\}$,
where $i<j$.

{\bf Subcase 2.1.}~~$k<i<j$. Then $T=\{u_iw_k,u_iw_j,u_jw_j\}$ is a rainbow $S$-tree.

{\bf Subcase 2.2.}~~$i\leq k\leq j$. Then $T=\{u_iw_k,u_jw_k\}$ is a rainbow $S$-tree.

{\bf Subcase 2.3.}~~$i<j<k$. Then $T=\{u_iw_i,u_jw_i,u_jw_k\}$ is a rainbow $S$-tree.
\end{proof}

With the aid of Theorem \ref{thm5}, we are now able to determine the
3-rainbow index of complete $t$-partite graph $K_r^t$. Note that we
always have $t\geq 3$. When $r=1$, $rx_3(K_1^t)=rx_3(K_t)$, which
was given in \cite{Zhang}.

\begin{thm}\label{thm6}
Let $K_r^t$ be a complete $t$-partite graph, where $r\geq 2$ and
$t\geq 3$. Then $rx_3(K_r^t)=3$.
\end{thm}
\begin{proof}
Let $U_1,U_2,\cdots,U_t$ be the $t$ partite sets of $K_r^t$, where
$|U_i|=r$. Suppose that $U_i=\{u_{i1},\cdots,u_{ir}\}$. If
$S\subseteq U_i$ and $|S|=3$, then every $S$-tree has size at least
3, hence $rx_3(K_{r,r})\geq 3$.

Next we show that $rx_3(K_r^t)\leq 3$.  We define a coloring $c$:
$E(K_r^t)\rightarrow \{1,2,3\}$ as follows.
\begin{equation}
c(u_{ai}u_{bj})=
   \begin{cases}
      1, & 1\leq i=j\leq r; \\
     2, & 1\leq i< j\leq r; \\
    3, & 1\leq j< i\leq r,
    \end{cases}
\end{equation}
where $1\leq a<b\leq t$.

We now show that $c$ is a 3-rainbow coloring of $K_r^t$. Let $S$ be
a set of three vertices of $K_r^t$.

{\bf Case 1.}~~The vertices of $S$ belong to the same partition.
Without loss of generality, let $S=\{u_{a1},u_{a2},u_{a3}\}$. Then
$T=\{u_{a1}u_{b2},u_{a2}u_{b2},u_{a3}u_{b2}\}$ is a rainbow
$S$-tree.

{\bf Case 2.}~~Two vertices of $S$ belong to the same partition.
Without loss of generality, let $S=\{u_{ai},u_{aj},u_{bk}\}$. If
$k<i<j$, then $T=\{u_{ai}u_{bk},u_{ai}u_{bj},u_{aj}u_{bj}\}$ is a
rainbow $S$-tree. If $i\leq k\leq j$, then
$T=\{u_{ai}u_{bk},u_{aj}u_{bk}\}$ is a rainbow $S$-tree. If $i<j<k$,
then $T=\{u_{ai}u_{bi},u_{aj}u_{bi},u_{aj}u_{bk}\}$ is a rainbow
$S$-tree.

{\bf Case 3.}~~Each vertex of $S$ belongs to distinct partite set.
Let $S=\{u_{ai},u_{bj},u_{ck}\}$, $a< b< c$.

{\bf Subcase 3.1}~~$i=j=k$. Without loss of generality, let
$S=\{u_{a1},u_{b1},u_{c1}\}$, then
$T=\{u_{a1}u_{b1},u_{a1}u_{b2},u_{b2}u_{c1}\}$ is a rainbow
$S$-tree.

{\bf Subcase 3.2}~~$i=j\neq k$. Without loss of generality, let
$S=\{u_{a1},u_{b1},u_{c2}\}$, then $T=\{u_{a1}u_{b1},u_{b1}u_{c2}\}$
is a rainbow $S$-tree.

{\bf Subcase 3.3}~~$i\neq j\neq k$. Without loss of generality, let
$S=\{u_{a1},u_{b2},u_{c3}\}$, then
$T=\{u_{a1}u_{c1},u_{c1}u_{b2},u_{b2}u_{c3}\}$ is a rainbow
$S$-tree.
\end{proof}

Another well-known class of graphs are the wheels. For $n\geq 3$,
the $wheel$ $W_n$ is a graph constructed by joining a vertex $v$ to
every vertex of a cycle $C_{n}:v_1,v_2,\cdots,v_{n},v_{n+1}=v_1$.
Given an edge-coloring $c$ of $W_n$, for two adjacent vertices $v_i$
and $v_{i+1}$, we define an edge-coloring of the graph by
identifying $v_i$ and $v_{i+1}$ to a new vertex $v'$ as follows: set
$c(vv')=c(vv_{i+1})$, $c(v_{i-1}v')=c(v_{i-1}v_i)$,
$c(v'v_{i+2})=c(v_{i+1}v_{i+2})$, and keep the coloring for the
remaining edges. We call this coloring the $identified$-$coloring$
at $v_i$ and $v_{i+1}$. Next we determine the $3$-rainbow index of
wheels.

\begin{thm}\label{thm7}
For $n\geq 3$, the $3$-rainbow index of the wheel $W_n$ is
\begin{equation}
 rx_3(W_n)=
   \begin{cases}
     2, & n=3; \\
     3, & 4\leq n\leq 6;\\
     4, & 7\leq n\leq 16;\\
     5, & n\geq 17.
    \end{cases}
\end{equation}
\end{thm}
\begin{proof}
Suppose that $W_{n}$ consists of a cycle
$C_{n}:v_1,v_2,\cdots,v_{n},v_{n+1}=v_1$ and another vertex $v$
joined to every vertex of $C_{n}$.

Since $W_3=K_4$, it follows by Lemma \ref{lem1} that $rx_3(W_3)=2$.

If $n=6$, let $S=\{v_1,v_2,v_4\}$. Since every $S$-tree has size at
least 3, $rx_3(W_6)\geq 3$. Next we show that $rx_3(W_6)\leq 3$ by
providing a rainbow 3-coloring of $W_6$ as follows:

\begin{equation}
c(e)=
   \begin{cases}
      1, & \text{if~e$\in \{vv_1,vv_4,v_2v_3,v_5v_6\}$}; \\
     2, & \text{if~e$\in \{vv_2,vv_5,v_3v_4,v_1v_6\}$}; \\
    3, & \text{if~e$\in \{vv_3,vv_6,v_4v_5,v_1v_2\}$}
    \end{cases}
\end{equation}

If $n=5$, since $|W_5|=6$, by Lemma \ref{lem2}, $rx_3(W_5)\geq 3$.
Then we show that $rx_3(W_5)\leq 3$. We provide a rainbow 3-coloring
of $W_5$ obtained from the rainbow 3-coloring of $W_6$ by the
identified-coloring at $v_5$ and $v_6$.

If $n=4$, by Theorem \ref{thm3}, $rx_3(W_4)\geq 3$. Then we show
that $rx_3(W_4)\leq 3$. We provide a rainbow 3-coloring of $W_4$
obtained from the rainbow 3-coloring of $W_6$ by the
identified-coloring at $v_5$ and $v_6$, $v_4$ and $v_5$,
respectively.

{\bf Claim 1.}~~ If $7\leq n \leq 16$, $rx_3(W_n)=4$.

First we show that $rx_3(W_7)\geq 4$. Assume, to the contrary, that
$rx_3(W_7)\leq 3$. Let $c:E(W_7)\rightarrow \{1,2,3\}$ be a rainbow
3-coloring of $W_7$. Since $d(v)=7>2\times 3$, there exists
$A\subseteq V(C_n)$ such that $|A|=3$ and all edges in $\{uv: u\in
A\}$ are colored the same. Thus, there must exist at least two
vertices $v_i,v_j\in A$ such that $d_{C_7}(v_i,v_j)\geq 2$ and a
vertex $v_k\in C_7$ such that $v_k\notin
\{v_{i-1},v_{i+1},v_{j-1},v_{j-1}\}$. Let $S=\{v_i,v_j,v_k\}$. Note
that the only $S$-tree of size 3 is $T=vv_i\cup vv_j\cup vv_k$, but
$c(vv_i)=c(vv_j)$, it follows that there is no rainbow $S$-tree,
which is a contradiction. Similarly, we have $rx_3(W_n)\geq 4$ for
all $n \geq 8$.

Second, we show that $rx_3(W_{16})\leq 4$, which we establish by
defining a rainbow 4-coloring $c$ of $W_{16}$ shown in Figure 1. It
is easy to check that $c$ is a rainbow 4-coloring of $W_{16}$.
Therefore, $rx_3(W_{16})=4$.

\begin{figure}[h,t,b,p]
\begin{center}
\scalebox{1}[1]{\includegraphics{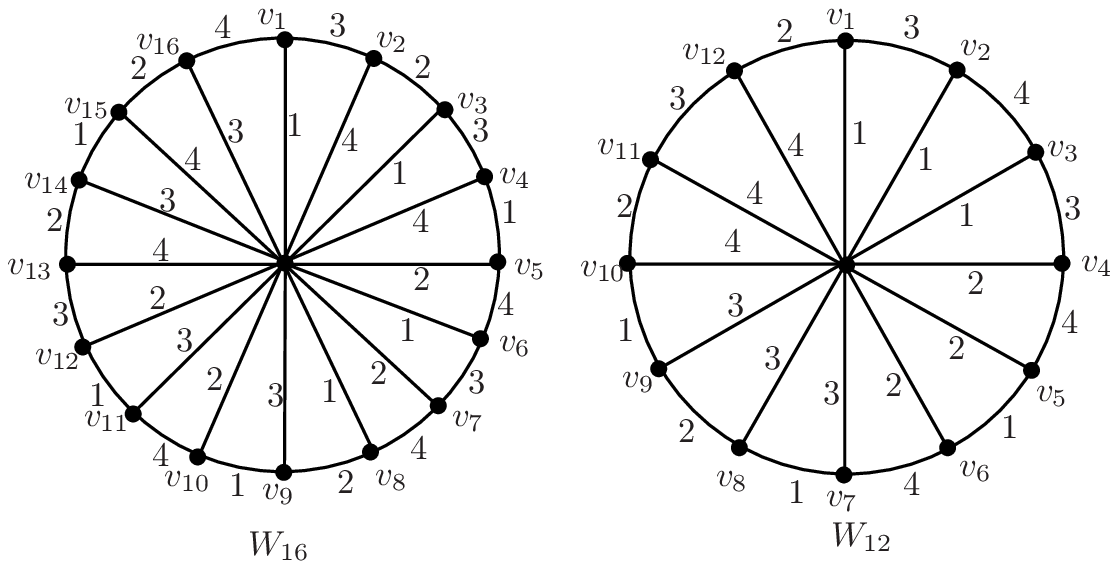}}\\[15pt]

Figure~1. 3-rainbow coloring of $W_{16}$ and $W_{12}$
\end{center}
\end{figure}

When $13\leq n \leq 15$, we provide a rainbow 4-coloring of
$W_{15}$, $W_{14}$, $W_{13}$ from the rainbow 4-coloring $c$ of
$W_{16}$ by consecutively using the identified-colorings at $v_{1}$
and $v_{16}$, $v_{12}$ and $v_{13}$, $v_{8}$ and $v_{9}$.

When $n=12$, we define a rainbow 4-coloring of $W_{12}$ shown in Figure 1.

When $7\leq n \leq 11$, we provide a rainbow 4-coloring of $W_{11}$,
$W_{10}$, $W_{9}$, $W_{8}$, $W_{7}$ from the rainbow 4-coloring $c$
of $W_{12}$ by consecutively identified-colorings at $v_{1}$ and
$v_{2}$, $v_{4}$ and $v_{5}$, $v_{7}$ and $v_{8}$, $v_{10}$ and
$v_{11}$, $v_{11}$ and $v_{12}$.

{\bf Claim 2.}~~ If $n\geq 17$, $rx_3(W_{n})=5$.

First we show that $rx_3(W_{17})\geq 5$. Assume, to the contrary,
that $rx_3(W_{17})\leq 4$. Let $c:E(W_{17})\rightarrow \{1,2,3,4\}$
be a rainbow 4-coloring of $W_{17}$. Since $d(v)= 17>4\times 4$,
there exists $A\subseteq V(C_n)$ such that $|A|=5$ and all edges in
$\{uv: u\in A\}$ are colored the same, say 1. Suppose that
$A=\{v_{i_1},v_{i_2},v_{i_3},v_{i_4},v_{i_5}\}$, where $i_1\leq
i_2\leq i_3\leq i_4\leq i_5$. There exists $k$ such that
$d_{C_{17}}(v_{i_k},v_{i_{k+1}})\geq 3$, where $1\leq k\leq 4$. Let
$S=\{v_{i_k},v_{i_{k+1}}, v_{i_{k+3}}\}$. Since
$d_{C_{17}}(v_{i_k},v_{i_{k+3}})\geq 2$ and
$d_{C_{17}}(v_{i_{k+1}},v_{i_{k+3}})\geq 2$, the only possible
$S$-tree is the path
$P=v_{i_{k+1}}v_{i_{k+2}}v_{i_{k+3}}v_{i_{k+4}}v_{i_{k+5}}$, where
addition is performed modulo 5. Thus color 1 must appear in $P$ and
every edge of the path must have a distinct color. By symmetry, we
consider two cases. If $c(v_{i_{k+1}}v_{i_{k+2}})=1$. Suppose
$c(v_{i_{k+2}}v_{i_{k+3}})=2$, $c(v_{i_{k+3}}v_{i_{k+4}})=3$. There
exists a vertex $v_0$, where $c(vv_0)=2$ or 3, such that
$d(v_0,A)\geq 3$. It is easy to see that there is no rainbow
$\{v_0,v_{i_{k+2}},v_{i_{k+4}}\}$-tree.  If
$c(v_{i_{k+2}}v_{i_{k+3}})=1$. We can also find such a vertex $v_0$
such that there exists no $\{v_0,v_{i_{k+2}},v_{i_{k+3}}\}$-tree,
which is a contradiction.

To show that $rx_3(W_{n})\leq 5$ for $n \geq 17$, define a rainbow
5-coloring of $W_n$ as follows:
\begin{equation}
 c(e)=
   \begin{cases}
     j, & \text{$e=vv_i$ and $i\equiv j \ (mod~ 5)$, $1\leq j\leq 5$}; \\
     i+3, & \text{$e=v_iv_{i+1}$}.
    \end{cases}
\end{equation}
It is easy to see that $c$ is a 5-rainbow coloring of $W_n$.
Therefore, $rx_3(W_n)=5$ for $n\geq 17$.
\end{proof}

\section{ The 3-rainbow index of 2-connected and 2-edge-connected graphs}

In this section, we give a sharp upper bound of the 3-rainbow index
for 2-connected and 2-edge-connected graphs. We start with some
lemmas that will be used in the sequel.

\begin{lem}\label{lem3}
If $G$ is a connected graph and $\{H_1, H_2,\cdots, H_k\}$ is a
partition of $V(G)$ into connected subgraphs, then $rx_3(G)\leq
k-1+\sum_{i=1}^krx_3(H_i)$.
\end{lem}
\begin{proof}
Let $G'$ be a graph obtained from $G$ by contracting each $H_i$ to a
single vertex. Then $G'$ is a graph of order $k$, so $rx_3(G')\leq
k-1$. Given an edge-coloring of $G'$ with $k-1$ colors such that
$G'$ is 3-rainbow connected. Now go back to $G$, and color each edge
connecting vertices in distinct $H_i$ with the color of the
corresponding edge in $G'$. For each $i=1,2,\cdots,k$, we use
$rx_3(H_i)$ new colors to assign the edges of $H_i$ such that $H_i$
is 3-rainbow connected. The resulting edge-coloring makes $G$
3-rainbow connected. Therefore, $rx_3(G)\leq
k-1+\sum_{i=1}^krx_3(H_i)$.
\end{proof}

To $subdivide$ an edge $e$ is to delete $e$, add a new vertex $x$,
and join $x$ to the ends of $e$. Any graph derived from a graph $G$
by a sequence of edge subdivisions is called a $subdivision$ of $G$.
Given a rainbow coloring of $G$, if we subdivide an edge $e=uv$ of
$G$ by $xu$ and $xv$, then we assign $xu$ the same color as $e$ and
assign $xv$ a new color, which also make the subdivision of $G$
3-rainbow connected. Hence, the following lemma holds.

\begin{lem}\label{lem4}
Let $G$ be a connected graph, and $H$ be a subdivision of $G$. Then
$rx_3(H)\leq rx_3(G)+|H|-|G|$.
\end{lem}

The $\Theta$-$graph$ is a graph consisting of three internally
disjoint paths with common end vertices and of lengths $a$, $b$, and
$c$, respectively, such that $a\leq b\leq c$. Then $a+b+c=n+1$.

\begin{lem}\label{lem5}
Let $G$ be a $\Theta$-graph of order n. If $n\geq 7$,  then $rx_3(G)\leq n-3$.
\end{lem}

\begin{proof}
Let the three internally disjoint paths be $P_1, P_2, P_3$ with the
common end vertices  $u$ and $v$, and the lengths of $P_1, P_2, P_3$
are $a,b,c$, respectively, where $a\leq b\leq c$.

(1). $b\geq 3$. Then $c\geq b\geq 3,\ a\geq 1$. First, we consider
the graph $\Theta_1$ with $a=1$, $b=3$ and $c=3$. We color $uP_1v$
with color 3, $uP_2v$ with colors $2,3,1$, and $uP_3v$ with colors
$1,3,2$. The resulting coloring makes $\Theta_1$ rainbow connected.
Thus, $rx_3(\Theta_1)\leq 3=|\Theta_1|-3$. For a general
$\Theta$-graph $G$ with $b\geq 3$ and $n\geq 7$, it is a subdivision
of $\Theta_1$, hence by Lemma \ref{lem4}, $rx_3(G)\leq
rx_3(\Theta_1)+|G|-|\Theta_1| \leq |G|-3$.

(2). $a=1$, $b=2$. Then since $a+b+c=n+1\geq 8$, $c\geq 5$. Consider
the graph $\Theta_2$ with $a=1$, $b=2$ and $c=5$. We rainbow color
$uP_1v$ with color 4, $uP_2v$ with colors $1,3$, and $uP_3v$ with
colors $2,3,4,2,1$. Thus, $rx_3(\Theta_2)\leq 4=|\Theta_2|-3$. For a
general $\Theta$-graph $G$ with $a=1$, $b=2$, $c\geq 5$, it is a
subdivision of $\Theta_2$, hence by Lemma \ref{lem4}, $rx_3(G)\leq
rx_3(\Theta_2)+|G|-|\Theta_2| \leq |G|-3$.

(3). $a=2$, $b=2$, Then since $a+b+c=n+1\geq 8$, $c\geq 4$. Consider
the graph $\Theta_3$ with $a=2$, $b=2$ and $c=3$. We rainbow color
$uP_1v$ with colors $3,2$, $uP_2v$ with colors $2,1$, and $uP_3v$
with colors $1,2,3$. Thus, $rx_3(\Theta_3)\leq 3=|\Theta_3|-3$. For
a general $\Theta$-graph $G$ with $a=2$, $b=2$, $c\geq 4$, it is a
subdivision of $\Theta_3$, hence by Lemma \ref{lem4}, $rx_3(G)\leq
rx_3(\Theta_3)+|G|-|\Theta_3| \leq |G|-3$.

Every $\Theta$-graph with $n\geq 7$ is one of the above cases,
therefore $rx_3(G)\leq n-3.$
\end{proof}

A 3-$sun$ is a graph $G$ which is defined from $C_6=v_1v_2\cdots
v_6v_1$ by adding three edges $v_2v_4$, $v_2v_6$ and $v_4v_6$.

\begin{lem}\label{lem6}
Let $G$ be a 2-connected graph of order 6. If $G$ is a spanning
subgraph of a 3-sun, then $rx_3(G)=4$. Otherwise, $rx_3(G)=3$.
\end{lem}
\begin{proof}
Since $G$ is a 2-connected graph of order 6, $G$ is a graph with a
cycle $C_6=v_1v_2\cdots v_6v_1$ and some additional edges.

If $G$ is a subgraph of a 3-sun, then since every tree connecting
the three vertices $\{v_1,v_3,v_5\}$ must have size at least 4,
which implies that $rx_3(G)\geq 4$. On the other hand, $rx_3(G)\leq
rx_3(C_6)\leq 4$. Therefore, $rx_3(G)=4$.

If there is an edge between the two antipodal vertices of $C_6$,
then by Lemma \ref{lem5}, $rx_3(G)=3$.

If $G$ contains the edges $v_1v_3$ and $v_2v_6$, then it contains
$\Theta_3$, defined in Lemma \ref{lem5}, as a spanning subgraph,
thus $rx_3(G)=3$.

If $G$ contains the edges $v_1v_5$ and $v_2v_4$, we give a rainbow
3-coloring $c$ of $G$: $c(v_1v_2)=c(v_4v_5)=1$,
$c(v_2v_3)=c(v_2v_4)=c(v_1v_5)=c(v_5v_6)=2$,
$c(v_3v_4)=c(v_1v_6)=3$.
\end{proof}

Let $H$ be a subgraph of a graph $G$. An $ear$ of $H$ in $G$ is a
nontrivial path in $G$ whose ends are in $H$ but whose internal
vertices are not. A nested sequence of graphs is a sequence
$\{G_0,G_1,\cdots,G_k\}$ of graphs such that $G_i\subset G_{i+1}$,
$0\leq i< k$. An $ear~ decomposition$ of a 2-connected graph $G$ is
a nested sequence $\{G_0,G_1,\cdots,G_k\}$ of 2-connected subgraphs
of $G$ such that: (1) $G_0$ is a cycle; (2) $G_i=G_{i-1}\cup P_i$,
where $P_i$ is an ear of $G_{i-1}$ in $G$, $1\leq i\leq k$; (3)
$G_k=G$. We call an ear decomposition $nonincreasing$ if
$\ell(P_1)\geq \ell(P_2)\geq \cdots \geq \ell(P_k)$, where
$\ell(P_i)$ denotes the length of $P_i$.

\begin{thm}\label{thm6}
Let $G$ be a 2-connected graph of order $n \ (n\geq 4)$. Then
$rx_3(G)\leq n-2$, with equality if and only if $G=C_n$ or $G$ is a
spanning subgraph of 3-sun or $G$ is a spanning subgraph of $K_5-e$
or $G$ is a spanning subgraph of $K_4$.
\end{thm}

\begin{proof}
Since $G$ is 2-connected, G contains a cycle. Let $C$ be the largest
cycle of $G$, then $|C|\geq 4$, $rx_3(C)\leq |C|-2$. Let $H_1=C$,
$H_2, H_3,\cdots, H_{n-|C|+1}$ be subgraphs of $G$, each is a single
vertex, then by Lemma \ref{lem3}, $rx_3(G)\leq n-|C|+rx_3(H_1)\leq
n-2$.

If $G=C$, then by Theorem \ref{thm1}, $rx_3(G)=n-2.$

If $G\neq C$, then $G$ contains a nonincreasing ear decomposition
$\{G_0,G_1,\cdots,G_k\}$. Let $H_1=C\cup P_1$, then $H_1$ is a
$\Theta$-graph. We choose $H_2, H_3,\cdots, H_{n-|H_1|+1}$ as
subgraphs of $G$ with a single vertex each, then by Lemma
\ref{lem3}, $rx_3(G)\leq n-|H_1|+rx_3(H_1)$.

If $|H_1|\geq 7$, then by Lemma \ref{lem5}, $rx_3(H_1)\leq |H_1|-3$,
hence $rx_3(G)\leq n-3$.

If $|H_1|=6$, we consider three cases.

\textbf{Case 1.} $|C|=6$. Then $\ell(P_1)=1$. Hence
$\ell(P_1)=\ell(P_2)= \cdots = \ell(P_k)=1$, $G$ is a graph of order
6. By Lemma \ref{lem6}, $rx_3(G)=4$ if and only if $G$ is a spanning
subgraph of a 3-sun.

\textbf{Case 2.} $|C|=5$. Then $\ell(P_1)=2$. Let $u$ and $v$ be the
end vertices of $P_1$. If $d_C(u,v)=1$, then we can find a cycle
larger than $C$, contradicting the choice of $C$. Otherwise,
$d_C(u,v)=2$, it is the graph $\Theta_3$ defined in Lemma
\ref{lem5}, then $rx_3(H_1)=rx_3(\Theta_3)\leq 3=|H_1|-3$, thus
$rx_3(G)\leq n-3$.

\textbf{Case 3.} $|C|=4$. Then $\ell(P_1)=3$. Let $u$ and $v$ be the
end vertices of $P_1$. Either  $d_C(u,v)=1$ or $d_C(u,v)=2$, we can
always find a cycle larger than $C$, a contradiction.

If $|H_1|=5$, there are two cases to be considered. If $|C|=5$, then
$\ell(P_1)=1$, hence $G$ is a graph of order 5. By Theorem
\ref{thm3}, $rx_3(G)=3=n-2$ except for $K_5$, whose 3-rainbow index
is 2. If $|C|=4$, then $\ell(P_1)=2$. Let $u$ and $v$ be the end
vertices of $P_1$. Note that $d_C(u,v)=2$. If $\ell(P_2)=1$, then
$G$ is a graph of order 5. If $\ell(P_2)\geq 2$, let $u'$ and $v'$
be the end vertices of $P_2$, then $\{u',v'\}=\{u,v\}$, otherwise,
we can find a cycle larger than $C$. Let $H_1'=H_1\cup P_2$, then
$H_1'$ is a graph consisting of 4 internally disjoint paths of
length 2 with common vertices $u$ and $v$. We assign the edges of
the four paths with colors 12, 21, 31, 13, the resulting coloring
makes $H_1'$ rainbow connected, thus, $rx_3(H_1')\leq 3=|H_1'|-3$.
Let $H_2', H_3',\cdots, H'_{n-|H'_1|+1}$ be subgraphs of $G$, each
is a single vertex, then by Lemma \ref{lem3}, $rx_3(G)\leq
n-|H_1'|+rx_3(H_1')\leq n-3$.

If $|H_1|=4$, then $|C|=4$, $\ell(P_1)=1$, $G$ is a graph of order
4, by Theorem \ref{thm3}, $rx_3(G)=2=n-2$.

Therefore, $rx_3(G)=n-2$ if and only if $G=C_n$ or $G$ is a spanning
subgraph of 3-sun or $G$ is a spanning subgraph of $K_5-e$ or $G$ is
a spanning subgraph of $K_4$.
\end{proof}

Now we turn to 2-edge-connected graphs. We call an ear is $closed$
if its endvertices are identical, otherwise, it is $open$. An open
or closed ear is called a $handle$. For a 2-edge-connected graph
$G$, there is a handle-decomposition, that is a sequence
$\{G_0,G_1,\cdots,G_k\}$ of graphs such that: (1) $G_0$ is a cycle;
(2) $G_i=G_{i-1}\cup P_i$, where $P_i$ is a handle of $G_{i-1}$ in
$G$, $1\leq i\leq k$; (3) $G_k=G$. Similar to Theorem \ref{thm6}, we
give an upper bound of 2-edge-connected graphs.

\begin{figure}[h,t,b,p]
\begin{center}
\scalebox{1}[1]{\includegraphics{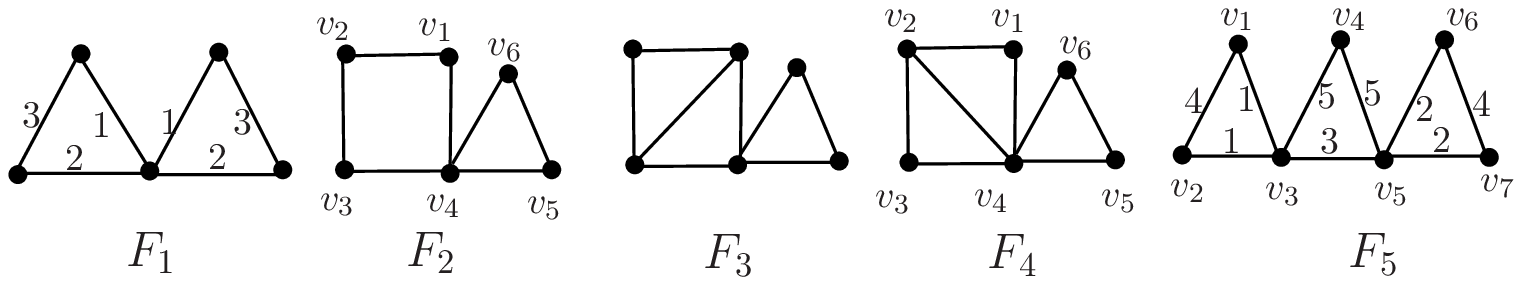}}\\[15pt]

Figure~2. Graphs with $rx_3(G)=n-2$
\end{center}
\end{figure}

\begin{figure}[h,t,b,p]
\begin{center}
\scalebox{1}[1]{\includegraphics{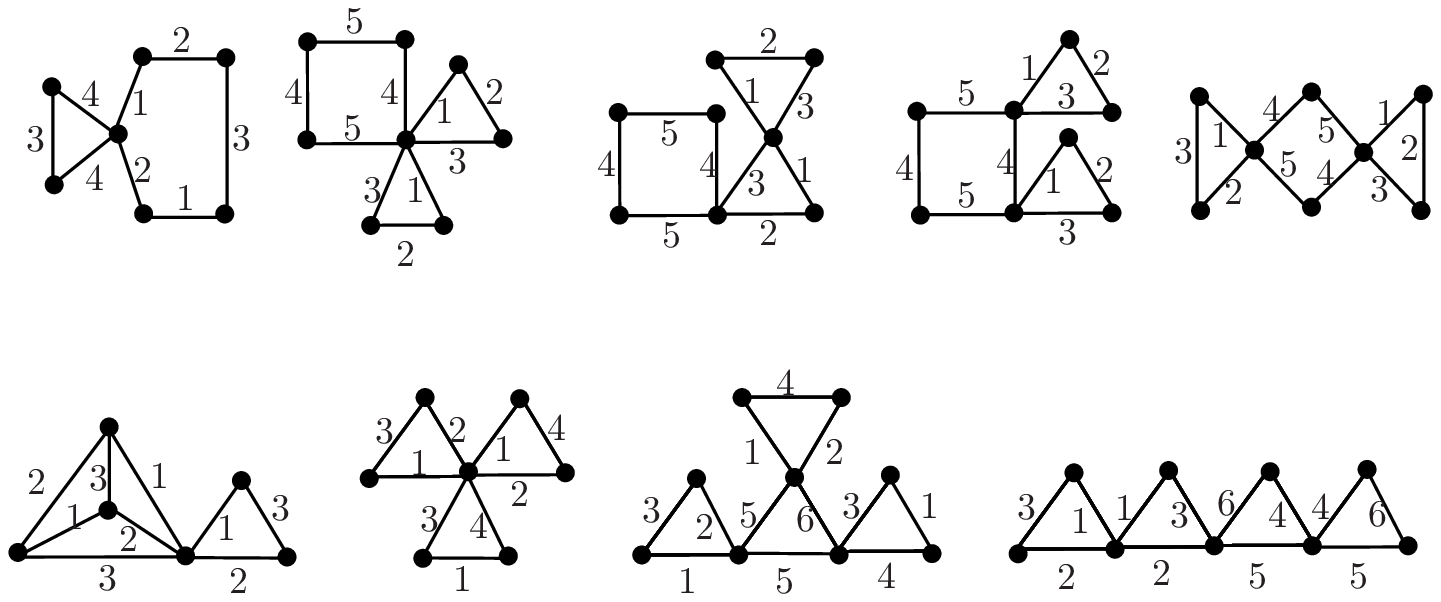}}\\[15pt]

Figure~3. Graphs with $rx_3(G)\leq n-3$
\end{center}
\end{figure}

\begin{thm}\label{thm8}
Let $G$ be a 2-edge-connected graph of order $n\geq 4$. Then
$rx_3(G)\leq n-2$, with equality if and only if $G$ is a graph
attaining the upper bound in Theorem \ref{thm6} or a graph in Figure
2.
\end{thm}

\begin{proof}

Let $C$ be the largest cycle of $G$. If $|C|\geq 4$, then
$rx_3(C)\leq |C|-2$. Otherwise, all cycles of $G$ are of length 3.
Since $n\geq 4$, there are at least two triangles $C_1$ and $C_2$
with a common vertex $v$. Let $F_1=C_1\cup C_2$, we rainbow color
$F_1$ with three colors, see Figure 2(1), thus $rx_3(F_1)\leq
3=|F_1|-2$. Let $H_1=C$ or $F_1$, $H_2, H_3,\cdots, H_{n-|H_1|+1}$
be subgraphs of $G$ with a single vertex each, then by Lemma
\ref{lem3}, $rx_3(G)\leq n-|H_1|+rx_3(H_1)\leq n-2$.

Now we determine the graphs that obtain the upper bound $n-2$.

If $G=C$, then by Theorem \ref{thm1}, $rx_3(G)=n-2.$

If $G\neq C$, then $G$ contains a handle-decomposition
$\{G_0,G_1,\cdots,G_k\}$. Let $H_1\subseteq G$, $H_2, H_3,\cdots,
H_{n-|H_1|+1}$ be subgraphs of $G$ with a single vertex each, then
by Lemma \ref{lem3}, if we show that $rx_3(H_1)\leq |H_1|-3$, then
we have $rx_3(G)\leq n-3$.

If $|C|\geq 4$ and $P_1$ is an open ear, we come back to Theorem
\ref{thm6}. If $|C|=3$ and $P_1$ is an open ear, then a cycle is of
length larger than $C$, a contradiction.

If $|C|\geq 4$ and $P_1$ is a closed ear, then $G_1$ is a union of
two cycles $C_1=C$ and $C_2=P_1$. If both of the cycles are of
length at least 4, we rainbow color each cycle $C_i$ with $|C_i|-2$
colors, which makes $G_1$ 3-rainbow connected. So we assume that
$C_2$ is of length 3. If $C_1$ is of length 5, we rainbow color
$G_1$ by 4 colors, see Figure 3(1). If $C_1$ is of length greater
than 5, then it is the subdivision of the graph in the case of
$|C_1|=5$. For all the above three cases, we have $rx_3(G_1)\leq
|G_1|-3$. Let $H_1=G_1$, it follows that $rx_3(G)\leq n-3$.

So it remains the case that $|C_1|=4,$ $|C_2|=3$, we denote this
graph by $F_2$, see Figure 2(2). Then $F_2$ is a subdivision of
$F_1$, so $rx_3(F_2)\leq4$. On the other hand, consider
$S=\{v_2,v_5, v_6\}$, every $S$-tree has size at least 4, hence
$rx_3(F_2)=4=|F_2|-2$. Observe that $P_2$ is a closed ear of length
at most 4, then $G_2=F_2\cup P_2$. If $\ell(P_2)=4$, then $G_2$
contains two cycles of length 4. If $\ell(P_2)=3$, we rainbow colors
$G_2$ with $|G_2|-3$ colors, see Figure 3(2-5). For the above two
cases, $rx_3(G_2)\leq |G_2|-3$. Let $H_1=G_2$, it implies that
$rx_3(G)\leq n-3$. If $\ell(P_2)=1$, then $P_2$ must be an edge
joining the vertices of $C_1$, there are two graphs, denoted by
$F_3$ and $F_4$. Similar to $F_2$, we have $rx_3(F_3)=|F_3|-2$. For
$F_4$, $rx_3(F_4)\leq rx_3(F_2)\leq4$. On the other hand, suppose
$rx_3(F_4)\leq3$. Consider $\{v_1,v_3, v_5\}$, $\{v_1,v_3, v_6\}$,
we have that $c(v_4v_6)= c(v_4v_5)$, which implies that there is no
rainbow $\{v_1,v_5,v_6\}$-tree or $\{v_1,v_5,v_6\}$-tree, a
contradiction. Hence $rx_3(F_4)=4=|F_4|-2$. Observe that $P_3$ is of
length 1, $G_3=F_3\cup P_3$ or $F_4\cup P_3$, we can rainbow color
$G_3$ by 3 colors, see Figure 3(6). Let $H_1=G_3$, then $rx_3(G)\leq
n-3$.

If $|C|=3$ and $P_1$ is a closed ear, then $\ell(P_1)=3$. Thus
$G_1=F_1$, it is easy to get $rx_3(G_1)= |G_1|-2$. If $P_2$ exists,
then it must be a closed ear of length 3, there are two cases for
the graph $G_2$. If $G_2$ is as Figure 3(7), then $rx_3(G_2)\leq
|G_2|-3$, let $H_1=G_2$, thus $rx_3(G)\leq n-3$. If $G_2$ is as
Figure 2(5), we prove that its 3-rainbow index is $|G_2|-2$. From
Figure 2(5), we have that $rx_3(G_2)\leq 5$. If $rx_3(G_2)\leq 4$,
let $c: E(G)\rightarrow \{1,2,3,4\}$ be the rainbow 4-coloring of
$G_2$. Consider $\{v_1,v_4,v_6\}$ and $\{v_1,v_4,v_7\}$, we have
$c(v_1v_3)\neq c(v_5v_6)$, $c(v_1v_3)\neq c(v_5v_7)$. If
$c(v_5v_6)=c(v_5v_7)$, suppose that $c(v_5v_6)=1$, $c(v_1v_3)=2$,
consider $\{v_1,v_6,v_7\}$, we may assume $c(v_3v_5)=3$,
$c(v_6v_7)=4$. Consider $\{v_2,v_6,v_7\}$, $\{v_1,v_2,v_6\}$,
$\{v_1,v_2,v_4\}$, $\{v_1,v_4,v_6\}$, we have $c(v_2v_3)=2,\
c(v_1v_2)=4$, $c(v_3v_4)=4, c(v_4v_5)=1$, but then there is no
rainbow tree connecting  $\{v_4,v_6,v_7\}$. If $c(v_5v_6)\neq
c(v_5v_7)$, then $c(v_1v_3)\neq c(v_2v_3)$, let $c(v_1v_3)=1,\
c(v_2v_3)=2,\ c(v_5v_6)=3,\ c(v_5v_6)=4$. Consider
$\{v_1,v_4,v_6\}$, then the colors 2 and 4 must appear in the
triangle $v_3v_4v_5$. Consider $\{v_2,v_4,v_7\}$, then the colors 1
and 3 must appear in the triangle $v_3v_4v_5$, which is impossible.
So we consider $P_3$, if it exists, then it must be a close ear,
there are two cases, no matter which case presents, we can give a
rainbow coloring with $|G_3|-3$ colors, see Figure 3(8-9). Let
$H_1=G_3$, then $rx_3(G)\leq n-3$.

Combining all the above cases, $rx_3(G)=n-2$ if and only if $G$ is a
graph attaining the upper bound in Theorem \ref{thm6} or a graph in
Figure 2.
\end{proof}

\end{document}